\documentclass[reqno]{amsart}
\usepackage{amsmath}
\usepackage{dsfont}
\usepackage{amsfonts}
\usepackage{amssymb}
\usepackage{aliascnt}
\usepackage{commath}
\usepackage{mathrsfs}
\usepackage{enumerate}
\usepackage{comment}
\usepackage[bookmarks=false,pdfstartview=FitH, pdfborder={0 0 0}, colorlinks=true,citecolor=blue, linkcolor=blue]{hyperref}
\newtheorem{theorem}{Theorem}[section]
\newaliascnt{conj}{theorem}
\newaliascnt{cor}{theorem}
\newaliascnt{lemma}{theorem}
\newaliascnt{fact}{theorem}
\newaliascnt{claim}{theorem}
\newaliascnt{prop}{theorem}
\newaliascnt{definition}{theorem}

\newtheorem{lemma}[lemma]{Lemma}

\newtheorem{prop}[prop]{Proposition}

\newtheorem{definition}[definition]{Definition}

\aliascntresetthe{conj} \aliascntresetthe{cor}
\aliascntresetthe{lemma} \aliascntresetthe{fact}
\aliascntresetthe{claim} \aliascntresetthe{prop}
\aliascntresetthe{definition}

\theoremstyle{definition}
\newaliascnt{example}{theorem}

\aliascntresetthe{example}

\theoremstyle{remark}
\newaliascnt{rmk}{theorem}
\newtheorem{remark}[rmk]{Remark}
\aliascntresetthe{rmk} 

\def\sek~{\S{}}

\numberwithin{equation}{section}

\newcommand{\diam}{\operatorname{diam}}
\newcommand{\rvol}{\operatorname{vol}}

\newcommand{\dist}{\operatorname{d}}

\newcommand{\otw}{{\rm width}_{\hat{\mathcal{T}}}}
\newcommand{\area}{\operatorname{area}}

\newcommand{\RR}{\mathds{R}}
\newcommand{\CC}{\mathds{C}}

\newcommand{\NN}{\mathds{N}}
\renewcommand{\SS}{\mathbf{S}}

\newcommand{\TT}{\mathds{T}_{Cl}}
\newcommand{\TTT}{\mathds{T}}

\newcommand{\foc}{{\rm rad}^{\bigodot}}

\begin{document}
\title{Gehring Link Problem, Focal Radius  and Over-torical width}
\author[J.~Ge]{Jian Ge}
\address[Ge]{Beijing International Center for Mathematical Research, Peking University, Beijing 100871, P. R. China.}
\email{jge@math.pku.edu.cn}
\thanks{*Partially supported by NSFC}
\maketitle
\begin{abstract}
In this note, we study the Gehring link problem in the round sphere, which motives our study of the width of a band in positively curved manifolds. Using the same idea, we are able to get a sphere theorem for hypersurface in  in the round $\SS^{n}$ in terms of its focal radius as well as the rigidity of Clifford hypersurface in $\SS^{n}$. The $3$-dimension case of our theorems confirm two conjectures raised by Gromov in \cite{Gro2018}.
\end{abstract}
\section{Introduction}
Let $(Y=\TTT^{2}\times [0, 1], g)$ be a Riemannian torical band, where $\TTT^{2}$ is the $2$-dimensional torus. In \cite{Gro2018}, Gromov studied the width of $Y$, which by definition is the distance between two boundary components of $Y$, under a lower scalar curvature bound of $g$. He also conjectured the upper bound of width of such band isometrically embedded in the round $3$-sphere is $\pi/2$. A related conjecture posed in \cite{Gro2018} is the focal radius of an embedded torus in the round $n$-sphere. 

Our starting point is to find the geometric intuition behind the conjectural upper bound $\pi/2$ of the width. Basically all the proofs in this paper can be summarized as:\\ 

``\emph{If two sets are linked in the sphere, they cannot be more than $\pi/2$-apart.}'' \\

The following example serves as our motivation. We take $\TT\subset \SS^{3}$, where $\SS^{3}:=\{(z_{1}, z_{2})\in \CC^{2}| |z_{1}|^{2}+|z_{2}|^{2}=1\}$ is the unit $3$-sphere and  the 2-dimensional Clifford Torus is defined by
$$
\TT^{2}:=\{(z_{1}, z_{2})\in \CC^{2}| |z_{1}|=|z_{2}|=1/\sqrt{2}\}.
$$ 
Let $B(\TT^{2}, r)$ be the $r$-tubular neighborhood of $\TT^{2}$ inside $\SS^{3}$. It is clear that $\SS^{3}\setminus B(\TT^{2}, r)$ has two connected component if $r<\pi/4$, each of which is a solid torus. These tori form a Hopf link in $\SS^{3}$. Their distance is clearly $\le \pi/2$. This reminds us the classical Gehring linking problem in $\RR^{3}$.
\begin{theorem}[\cite{Ort1975}, \cite{BS1983}]
Let $A$ and $B$ are closed curves which are differentiably embedded in $\RR^{3}$ such that they are linked and $\dist(A, B)\ge 1$, then the length of $A$ and $B$ must great or equal to $2\pi$. 
\end{theorem}

We now state our first result. 
\begin{theorem}[Gehring Link Problem in 3-sphere]\label{lem:glp}
Let $A, B$ be two disjoint linked Jordan curves in the unit $3$-sphere $\SS^{3}$. Then 
$$\dist(A, B)\le \pi/2,$$
with equality holds if and only if $A$ and $B$ are the dual great circles in the Hopf fibration.
\end{theorem}

Our method of the proof can be easily generlized to higher dimension. The classical Gehring Link problem in $\RR^{n}$ is proved in \cite{Gage1980}.
\begin{theorem}[Spherical Gehring Link Problem]\label{thm:gehring}
If $A^{k}$ and $B^{l}$ are $k$ and $l$ spheres which are embedded in $\SS^{n}$ such that they are linked, where $n=k+l+1$. Then
$$\dist(A, B)\le \pi/2,$$
with equality holds if and only if $A$ and $B$ are the dual great sub-spheres in $\SS^{n}$, i.e. $\SS^{n}=A*B$, where $*$ denotes the spherical join.
\end{theorem}

\begin{remark}
Our proof of \autoref{thm:gehring} remains valid if we replace the $n$-sphere by a closed Riemannian manifold $(M^{n}, g)$ with sectional curvature $\ge 1$ or more general Alexandrov spaces with lower curvature bound $1$ (except for the rigidity part, which should be stated as the spherecial join of $A$ and $B$). For simplicity we focus only on the $n$-sphere.
\end{remark}

\autoref{lem:glp} explains heuristically why the width of a torical band cannot be too large: The complement of the band in $\SS^{3}$ are `linked'. However, this is only partially correct. In general the complement of a torical band is not a classical link. For example when the torus is the boundary of a nontrivial knot in $\SS^{3}$, only one component of the complement is homotopic to a cicle. Nevertheless, we can still make use the idea of `link' as an obstruction to have large width, we call this obstruction `boundary irreducible' see \autoref{def:birr}. In fact, using this idea we prove the following theorem, which confirms a conjecture of Gromov in \cite{Gro2018}.
\begin{theorem}\label{thm:2}
The over-torical width of $\SS^{3}$ denoted by $\otw(\SS^{3})$, is $\pi/2$.
\end{theorem}

Another estimate can be drawn from the proof of \autoref{thm:gehring} is the estimate of focal radius of torus in $\SS^{3}$. Let $\Sigma^{2}\subset \SS^{3}$ be a smoothly embedded closed surface. Recall that the focal radius of $\Sigma\subset \SS^{3}$ is the largest number $r>0$ such that
$$
\exp: \{(p, v)\in \nu \Sigma| p\in \Sigma, |v|<r\}\to \SS^{3}
$$
is a diffeomorphism onto its image, where $\nu \Sigma$ is the normal bundle of $\Sigma$. The focal radius of $\Sigma\subset \SS^{3}$ will be denoted by $\foc(\Sigma)=\foc(M\subset \SS^{3})$. By the standard comparison argument, it is well known that any hypersurface $\Sigma$ in $\SS^{3}$ has focal radius $\le \pi/2$. In fact if a closed surface $\Sigma^{2}$ in $\SS^{3}$ has focal radius equal to $\pi/2$ then $\Sigma^{2}$ has to be the equatorial 2-sphere, cf. \cite{GW2018}. The following question is asked in \cite{Gro2018}: Let $\Sigma^{n}$ be a smoothly embedded $n$-torus in $\SS^{2n-1}$, what is the largest possible focal radius? Gromov conjectured that the Clifford Torus $\TT^{n}$ is the only torus realizes the conjectured upper bound $\arcsin(1/\sqrt{n})$. The upper bound for $n=2$ actually follows from Corollary C in \cite{GW2020}. We give a different proof of this result as well as the rigidity.

\begin{theorem}\label{thm:1}
Let $\Sigma^{2}$ be a smoothly embedded  $2$-torus in $\SS^{3}$. Then 
$$
\foc(\Sigma)\le \frac{\pi}{4},
$$
with equality holds if and only if $\Sigma$ is the Clifford Torus $\TT^{2}$.
\end{theorem}

In fact, we prove the following stronger result
\begin{theorem}\label{thm:1k}
Let $\Sigma^{n-1}$ be an orientable hypersurface embedded in $\SS^{n}$. Suppose $\Sigma$ is not homeomorphic to $\SS^{n-1}$, then $\foc(\Sigma)\le \pi/4$, with equality holds if and only if $\Sigma$ is isometric to the Clifford hypersurface $\SS^{k}(1/\sqrt 2)\times \SS^{l}(1/\sqrt 2)$ for some$k, l\in \NN$ such that $n=k+l+1$.
\end{theorem}

Next theorem shows that spheres are the only hypersurfaces with large focal radius in positively cruved manifold.
\begin{theorem}[A Topological Sphere Theorem]\label{thm:1h}
Let $(M^{n}, g)$ be a simply connected closed Riemannian manifold with $\sec\ge 1$ and $\Sigma^{n-1}$ be a smoothly embedded  orientable hypersurface in $M$. Suppose $\foc(\Sigma)> \frac{\pi}{4}$. Then $M$ is homeomorphic to $\SS^{n}$ and $\Sigma$ is homeomorphic to $\SS^{n-1}$.
\end{theorem}

\begin{remark}
By $h$-cobordism theorem, $(n-1)$-dimensional exotic sphere does not embedded in $\RR^{n}$ for $n\ne 5$. Therefore $\Sigma$ is diffeomorphic to $\SS^{n-1}$ for $n\ne 5$ in \autoref{thm:1h}.
\end{remark}

In Section 1 we prove the Gehring Link problem in sphere. The ideas of the proof are developed further in Section 2, where we provide the proofs of  \autoref{thm:1h} and \autoref{thm:2}.  Another short geometric proof of \autoref{thm:1} is provided in Section 3. \\

\emph{Acknowledgement}: It is my great pleasure to thank Professor Misha Gromov for his comments and interests in our work and Professor Luis Guijarro for comments after reading the first draft of this paper. 

I would like to thank Professor Yuguang Shi for bringing J. Zhu's paper \cite{Zhu2020} to my attention where the author estimated the 3-dimension width using a completely different method.

\section{Proof of the Spherical Gehring Link Problem}
Let $A^{k}, B^{l}$ be two submanifold of $\SS^{n}$, we call $A$ and $B$ are \emph{unlinked} if there exists an embedded $(n-1)$-sphere $\SS^{n-1}$ in $\SS^{n}$ such that $A$ and $B$ lie in complementary hemispheres; otherwise $A$ and $B$ are called \emph{linked}. Note that we do not require $k+l+1=n$ in the definition of `linked', it is only required if we want to define the `Linking number'.

The key idea goes back to Grove-Shiohama's proof of the Diameter Sphere Theorem, cf \cite{GS1977}. Namely we have the following
\begin{prop}\label{prop:convex}
Let  $(M, g)$ be a closed Riemannian manifold with sectional curvature $\sec(g)\ge 1$. Suppose $\diam(M, g) \ge \pi/2$. Then for any point $x\in M$ and $r>\pi/2$, the set $N:=M\setminus B(x, r)$ is either empty or a totally geodesic submanifold with strictly convex (possible empty) boundary.
\end{prop}

For any closed set $A\subset M$, it is clear that 
$$M\setminus B(A, r)=\cap_{x\in A}\{M\setminus B(x, r)\}.$$
It follows that $M\setminus B(A, r)$ is a totally geodesic submanifold with strictly convex (possible empty) boundary. Moreover if $\partial N$ is nonempty, $N$ is homeomorphic to the standard $n$-disk.

\begin{proof}[Proof of \autoref{thm:gehring} ]
Suppose $\dist(A, B)>\pi/2$, then $B\subset \SS^{n}\setminus B(A, \pi/2)$. If $A$ is not a great sub-sphere $\SS^{k}\subset \SS^{n}$, then the set $N:=\SS^{3}\setminus B(A, \pi/2)$ is a connected convex set with nonempty boundary. It follows that $N$ is homeomorphic to a disk. By smoothing the distance function $d_{A}(\cdot)$, we conclude $\partial N$ is diffeomorphic to the sphere $S^{n-1}$. But $B\subset \SS^{n}\setminus B(A, \pi/2)$. Therefore it contradict to the assumption that $A$ and $B$ are linked in $\SS^{3}$. It follows that $\dist(A, B)\le \pi/2$.

Suppose $\dist(A, B)=\pi/2$, then $A^{anti}:=\SS^{n}\setminus B(A, \pi/2)$ must has empty boundary. Since $A^{anti}$ is convex and in particular totally geodesic, $B$ is isometric to $\SS^{l}$ for some $l\in \{1,\cdots, (n-1)\}$. Apply the same discussion to $B$, we know $A$ is isometric to $\SS^{k}$ for some $k\in \{1,\cdots, (n-1)\}$. Since $A$ and $B$ are linked, $n\le k+l+1$.
\end{proof}

The proof can be carried verbatim to the case when the sphere $\SS^{n}$ is replaced by a closed Riemannian manifold $(M, g)$ with $\sec\ge 1$.

\section{The complement of a non-sphercial band are linked}
One crucial step in the proof of \autoref{thm:gehring}  in the previous section is to produce an embedded sphere $\SS^{n-1}$ in $\SS^{n}$ which separates $A$ and $B$,  provide they are more than $\pi/2$ apart. In particular this separating hypersphere represents a nontrivial element in $\pi_{n-1}(\SS^{n}\setminus \{A, B\})$. This put a restriction on the topology of the complement. In this section, we focus on the complement instead of the set $A$ and $B$.

Let's recall several definitions in \cite{Gro2018}. A \emph{proper band} is a connected manifold with two distinguished disjoint non-empty subset in the boundary $\partial{Y}_{-}, \partial{Y}_{+}$ such that
$\partial Y=\partial{Y}_{-}\cup \partial{Y}_{+}$. A proper band $Y$ is called \emph{over-torical} if  there exists a map 
$$
f: Y\to \underline{Y}:=\TTT^{n-1}\times [0, 1],
$$
with nonzero degree and respect the boundaries: $\partial Y_{\pm}\to \partial \underline{Y}_{\pm}$. We introduce the following definition:
\begin{definition}\label{def:birr}
Let $Y$ be a $n$-dimensional proper band. We call $Y$ is \emph{boundary reducible} if there exists an embedded $(n-1)$-sphere $S^{n-1}\subset Y$, such that $S$ separates $\partial Y_{-}$ from $\partial Y_{+}$. Otherwise $Y$ is called \emph{boundary irreducible}.
\end{definition}

\begin{remark}Note that if $Y$ is boundary irreducible, it is still possible to find an embedded $(n-1)$-sphere in $Y$ that does not bound a $n$-ball. Namely it is possible $Y$ is reducible in the classical sense. For example, take connected sum of a torical band with any $3$-manifold.
\end{remark}

For a closed orientable manifold $\Sigma^{n-1}$, let's consider the band $Y:=\Sigma^{n-1}\times [0, 1]$. Suppose that $Y$ is isometrically embedded in $(M^{n}, g)$, a closed Riemannian manifold with $\sec(M, g)\ge 1$. We will identify $Y$ with its image in $M$. Let $Y_{-}$ and $Y_{+}$ be the boundaries $Y\times \{0\}$ and $Y\times \{1\}$ of $Y$. 
Suppose $R:=\dist(Y_{-}, Y_{+})>\pi/2$. Let 
$$N:=M^{n}\setminus B(Y_{+}, R-\varepsilon),\ \ \ \varepsilon<(R-\pi/2)/10.$$
By \autoref{prop:convex}, it is clear that $Y_{-}\subset N$ and $N$ is homeomorphic to a disk, whence $\partial N$ is homeomorphic to $\SS^{n-1}$. In particular, we find an embedded $(n-1)$-sphere in $Y$ which separates $Y_{-}$ from $Y_{+}$. Namely $Y$ can be written as a connected sum of two manifolds, this is only possible if $\Sigma^{n-1}$ itself is homeomorphic to $\SS^{n-1}$. Therefore, we just proved
\begin{prop}\label{prop:sphere}
Let $(M^{n}, g)$ be a closed orientable manifold with $\sec\ge 1$. Let $Y=\Sigma^{n-1}\times [0, 1]$ be a band isometrically embedded in $M^{n}$. Suppose $\dist(Y_{-}, Y_{+})>\pi/2$, then $\Sigma$ is homeomorphic to $\SS^{n-1}$.
\end{prop}
Clearly \autoref{prop:sphere} and the Diameter Sphere Theorem (cf. \cite{GS1977}) imply \autoref{thm:1h} immediately.  Now we move to the study of the non-spherical band.
\begin{proof}[Proof of \autoref{thm:1k}]
We consider the set 
$$N:=\SS^{n}\setminus B(\Sigma, \pi/4).$$
Since $\Sigma$ is orientable, the set $N$ has two connected components, denote them by $\{A, B\}$. Since 
$$
\dist(A, B)\ge \dist(A, \Sigma)+\dist(B, \Sigma)=\pi/2.
$$
Then there are two possibilities. One is $\diam(M, g)=\pi/2$. It follows from the diameter rigidity theorem of Gromoll-Grove that $(M, g)$ is isometric to a CROSS or homeomorphic to $\SS^{n}$. If $(M, g)$ is not heomorphic to a sphere, then $A$ and $B$ are the dual pairs of 

By our assumption that $\Sigma$ is not homeomorphic to $\SS^{n-1}$, it follows that $A$ and $B$ are both totally geodesic submanifolds in $\SS^{n}$ with empty boundary. Therefore $A$ and $B$ are great sub-spheres in $\SS^{n}$, namely, $A=\SS^{k}$ and $B=\SS^{l}$ for some $k, l\in \{1, \cdots, n-1\}$. 
it follows that $k+l+1\le n$. On the other hand 
$$
A=\SS^{n}\setminus (B(B, \pi/2)),
$$
and vice versa, then $k+l+1=n$. i.e. $A$ and $B$ are the dual great sub-spheres in $\SS^{n}$. It follows that $\Sigma$ is the Clifford hypersurface $\SS^{k}(1/\sqrt 2)\times \SS^{l}(1/\sqrt 2)$.
\end{proof}

Now we turn to the estimate of the over-toric width. The proof is similar as above. Since $\pi_{2}(\underline{Y})=0$, we have the following
\begin{lemma}\label{lem:irred}
Let $Y$ be a $3$-dimensional proper over-torical band, then $Y$ is boundary irreducible.
\end{lemma}

Recall that the over-torical width of $\SS^{3}$, denoted by $\otw(\SS^{3})$ is defined as the supremum of numbers $d$ such that there exists a proper over-torical band $Y$ of width $d$ and
an isometric immersion
$$
\phi: Y \to \SS^{3}.
$$
Since the $r$-neighborhood of Clifford Torus in $\SS^{3}$ provides a torus band of width arbitrarily close to $\pi/2$, $\otw(\SS^{3})\ge \pi/2$. To prove \autoref{thm:2}, it suffices to show $\otw(\SS^{3})\le \pi/2$. Let $Y$ be as above, we have
\begin{prop}
The width of $Y$ is less than or equal to $\pi/2$.
\end{prop}
\begin{proof}
Suppose the width $d:=\dist(Y_{-}, Y_{+})>\pi/2$, where we equipped $Y$ with the pull back of the round metric on the sphere. Consider the distance function $\rho: Y\to \mathbb R$, defined by
$$
\rho(x)=\dist(x, Y_{-}),
$$
where the distance $\dist$ is defined by the shortest curve connecting $x$ to $Y_{-}$. Since we make no assumption on the convexity of $Y_{\pm}$, such a curve might intersect the boundary $Y_{+}$. However, if the width $d=\dist(Y_{-}, Y_{+})>\pi/2$, we know for any point
$$
x\in \Sigma:=\rho^{-1}\left(\frac{\pi/2+d}{2}\right),
$$
any geodesic connecting $x$ to $Y_{-}$ that realizes $\rho(x)$ does not intersect $Y_{+}$. Since $Y$ has constant curvature $1$, the standard comparison argument shows that $\Sigma$ is locally strictly convex in $Y$. Therefore under the isometric immersion $\phi: Y\to \SS^{3}$, its image $\overline{\Sigma}:=\phi(\Sigma)\subset \SS^{3}$ is also locally strictly convex in $\SS^{3}$. By the classical theorem of Hadamard (cf. \cite{Had1897}), all such surfaces must be embedded, hence $\overline{\Sigma}$ is an embedded  $\SS^{2}$. Since $\phi$ is an immersion, it follows that $\phi$ is a covering map. Therefore $\Sigma$ is a disjoint union of embedded separating $2$-spheres in $Y$ (It might be disjoint only if $Y_{-}$ or $Y_{+}$ has more than one connected components). By connecting these spheres via cylinders we get an embedded separating $2$-sphere in $Y$. This contradict to \autoref{lem:irred}. Therefore $d\le \pi/2$.
\end{proof}

\section{Yet another proof of the \autoref{thm:1}}
In this section, we give another proof of \autoref{thm:1} based on Weyl's Tube formula and the solution of Willmore Conjecture. Let $\Sigma^{2}\subset \SS^{3}$ be an embedded $2$-torus with focal radius $r\in [\pi/4, \pi/2]$. It follows from the definition of focal radius that
\begin{equation}\label{eq:1}
\rvol(B(\Sigma, r))\leq\rvol(\SS^{3})= 2\pi^{2}.
\end{equation}
By the volume estimate \autoref{lem:tube}, which will be proved in next section, and the fact that $\sin(2r)\ge \cot(r)$ if $r\in [\pi/4, \pi/2]$, we have
\begin{equation}\label{eq:2}
\rvol(B(\Sigma, r))=\sin(2r)\area(\Sigma)\ge \cot(r)\area(\Sigma).
\end{equation}
Combining \eqref{eq:1} and \eqref{eq:2}, we have
\begin{equation}\label{eq:3}
\cot(r)\cdot\area(\Sigma)\leq 2\pi^{2}.
\end{equation}
On the other hand by the solution of Willmore Conjecture we have
\begin{equation}\label{eq:4}
\int_{\Sigma}1+\left( \frac{\kappa_{1}+\kappa_{2}}{2}\right )^{2}d\mu \ge 2\pi^{2},
\end{equation}
where $\kappa_{1}$ and $\kappa_{2}$ are the principal curvatures of $M$ with respect to the standard metric on $\SS^{3}$.
It follows by the Gauss-Bonnet formula that 
\begin{equation}\label{eq:5}
2\pi^{2}\le \int_{\Sigma}1+\left( \frac{\kappa_{1}+\kappa_{2}}{2}\right )^{2}d\mu\le \int_{\Sigma}\left(\frac{\kappa_{1}^{2}+\kappa_{2}^{2}}{2}\right) d\mu.
\end{equation}
Since the focal radius of $M$ equals to $r$, it follows that 
$$|\kappa_{i}|\le \cot(r)$$ 
for $i=1,2$. In fact this can be seen by touching each point $p\in \Sigma$ by a geodesic sphere with radius $r$ in $\SS^{3}$ from both sides of $\Sigma$, or we can use the theorem of \cite{GW2020}. Plugging them back to \eqref{eq:5} yields
\begin{equation}\label{eq:6}
2\pi^{2}\le \cot^{2}r\cdot \area(M).
\end{equation}
Using \eqref{eq:3}, we have
$$
2\pi^{2}\le \cot^{2}r\cdot\area(M)\le \cot r\cdot2\pi^{2},
$$
i.e. $\cot(r)\ge 1$, which implies $r=\pi/4$ since we assume $r\ge \pi/4$. The rigidity part of the theorem follows from the rigidity part of the Willmore Conjecture.

\section{Volume of the tube}
In this section, we calculate the volume of the $r$-neighborhood of $\Sigma^{2}\subset \SS^{3}$. It is a special case of Weyl's Tube Formula. In fact it implies the focal radius $\le \pi/4$ by monotonicity of $\sin(2r)$.
\begin{lemma}[Tube Formula]\label{lem:tube}
Let $\Sigma$ be an embedded $2$-torus in $\SS^{3}$. Then for any $r\le \foc(\Sigma)$, we have
$$
\rvol(B(\Sigma, r))=\sin(2r)\area(\Sigma).
$$
\end{lemma}
\begin{proof}
Let 
$$
\Sigma(t):=\{x\in \SS^{3}|x =\exp_{p}(tv),  p\in M, v\in \nu_{p}M\}
$$
be the parallel hypersurface with signed distance $t\in [-r, r]$ to $M$. For any $p\in \Sigma^{2}$ let $\kappa_{1}(p), \kappa_{2}(p)$ be the two principal curvature of $\Sigma$ at $p$. Using Fermi coordinate and co-area formula we have
\begin{align*}
\rvol(B(\Sigma, r))&=\int_{-r}^{r}\area(\Sigma(t)) dt\\
&=\int_{-r}^{r}\cos^{2}t\int_{\Sigma}(1-\kappa_{1}\tan t )(1-\kappa_{2}\tan t )d\mu dt\\
&=\int_{\Sigma}\int_{-r}^{r} (\cos^{2}t-(\kappa_{1}+\kappa_{2})\cos^{2}t\cdot\tan t+\kappa_{1}\kappa_{2}\sin^{2}t )dt d\mu\\
&=\int_{\Sigma}\int_{-r}^{r} (\cos^{2}t+\kappa_{1}\kappa_{2}\sin^{2}t )dt d\mu\\
&=\int_{\Sigma}\left(\int_{-r}^{r} (\cos^{2}t-\sin^{2}t)dt+\int_{-r}^{r} ((1+\kappa_{1}\kappa_{2})\sin^{2}t)dt\right)d\mu\\
&=\int_{\Sigma}\int_{-r}^{r} (\cos^{2}t-\sin^{2}t)dt d\mu\\
&=\sin(2r)\area(\Sigma).
\end{align*}
\end{proof}

\bibliographystyle{alpha}
\bibliography{mybib}

\bigskip
\noindent
{\small Jian Ge}\\
{\small Beijing International Center for Mathematical Research,\\
Peking University, Beijing 100871, China
} \\
{\small E-mail address: jge@math.pku.edu.cn}\\
\end{document}